\documentclass[12pt, reqno]{amsart}
\usepackage{amsmath, amsthm, amscd, amsfonts, amssymb, graphicx, color}
\usepackage[bookmarksnumbered, colorlinks, plainpages]{hyperref}
\hypersetup{colorlinks=true,linkcolor=red, anchorcolor=green, citecolor=cyan, urlcolor=red, filecolor=magenta, pdftoolbar=true}

\usepackage{tensor}

\newtheorem{theorem}{Theorem}[section]
\newtheorem{lemma}[theorem]{Lemma}

\theoremstyle{definition}
\newtheorem{definition}[theorem]{Definition}
\newtheorem{example}[theorem]{Example}

\newtheorem{remark}[theorem]{Remark}

\numberwithin{equation}{section}

\usepackage{fullpage}

\begin{document}
	
	\setcounter{page}{1}
	
	\title[PMS-completeness]{Partial quasi-metric completeness via Kannan-type fixed points}

	\author[Ya\'e Ulrich Gaba]{Ya\'e Ulrich Gaba$^{1,2,\dagger}$}

	\address{$^{1}$ Institut de Math\'ematiques et de Sciences Physiques (IMSP), 01 BP 613 Porto-Novo, B\'enin.}

	\address{$^{2}$ African Center for Advanced Studies (ACAS),
		P.O. Box 4477, Yaounde, Cameroon.}

	\email{\textcolor[rgb]{0.00,0.00,0.84}{yaeulrich.gaba@gmail.com
	}}


	
	\date{Received: xxxxxx; Accepted: zzzzzz.
		\newline \indent $^{\dagger}$Corresponding author}
	
	\begin{abstract}
		In this short note, we obtain partial quasi-metric versions of Kannan's fixed point theorem for
		self-mappings.  Moreover, we use these fixed points results
		to characterize a certain type of completeness in partial quasi-metric spaces. We have reported an example to support our result.
		\end{abstract} 
	
	\maketitle
	
	\section{Introduction and preliminaries}

In 1968, Kannan\cite{kan} proved the following fixed point theorem:	
	
\begin{theorem}
	Let $(X,d)$ be a complete metric space and $T$ be a self-mapping on $X$ satisfying

	 \begin{equation}\label{eq1}
	 d(Tx,Ty)\leq K\{d(x,Tx)+d(y,Ty)\}
	 \end{equation}
	 for all $x,y \in X$ and $K\in [0,\frac{1}{2})$.
	Then $T$ has a unique fixed point $z\in X$, and for any $x\in X$ the sequence of iterates $(T^nx)$ converges to $z$.
\end{theorem}

This result by Kannan is an extension of Banach contraction principle in the sense that it proves that there exists a contractive map accompanied with fixed point which is not
necessarily continuous. Kannan's theorem is important because
Subrahmanyam \cite{sub} proved that Kannan's theorem characterizes the metric completeness. That is, a metric space
$X$ is complete if and only if every mapping satisfying \eqref{eq1} on $X$ with constant $K<\frac{1}{2}$
has a fixed point. However contractions
(in the sense of Banach) do not have this property.

On the other hand, and motivated in part by the fact that partial quasi-metric spaces
provide suitable frameworks in several areas of asymmetric functional analysis,
domain theory, complexity analysis and in modelling partially defined
information, which often appears in computer science, the development of
the fixed point theory for theses spaces appears to be an interesting focus for current research. In this setting, the problem of characterizing
partial quasi-metric completeness via fixed point theorems arises in a natural way.  This
problem is indeed more interesting due to the fact that asymmetric structures present a natural inclination for different type of completeness.

Partial metrics were introduced by Matthews\cite{mat} in 1992. They generalize the concept of a metric space in the sense
that the self-distance from a point to itself need not be equal to zero. In \cite{kun}, by dropping the symmetry condition in the definition of a partial metric, K\"unzi et al. studied another variant of
partial metrics, namely partial quasi-metrics.

In concluding this introductory part, we recall some pertinent notions and properties on
partial quasi-metric spaces which will be useful later on.

\begin{definition} (Compare \cite{mat})
	A partial metric type on a set $X$ is a function $p: X \times X \to [0, \infty)$ such that:
	\begin{enumerate}
		\item[(pm1)] $x = y$ iff $(p(x, x) = p(x, y)=p(y,y)$ whenever $x, y \in X$,
		\item[(pm2)] $p(x, x)\leq p(x, y)$ whenever $x, y \in X$,
		
		\item[(pm3)] $p(x, y) = p(y, x);$ whenever $x, y \in X$,
		
		\item[(pm4)] $$p(x, y) + p(z,z)\leq  p(x,z)+p(z,y)$$ 
		for any points $x,y,z\in X$. 
		
	\end{enumerate}

	The pair $(X, p)$ is called a partial metric space.

\end{definition}
It is clear that, if $p ( x , y ) = 0$ , then, from (pm1) and (pm2), $x = y$.

\begin{definition} (\cite[Definition 1.]{kun})
	A partial quasi-metric on a set $X$ is a function $p: X \times X \to [0, \infty)$ such that:
	\begin{enumerate}
		\item[(1a)] $p(x, x)\leq p(x, y)$ whenever $x, y \in X$,
		\item[(1b)] $p(x, x)\leq p(y, x)$ whenever $x, y \in X$,
		\item[(2)] $p(x, z) + p(y, y) \leq ( p(x, y) + p(y, z))$ whenever $x, y,z \in X$, for some $K\geq 1,$
		\item[(3)] $x = y$ iff $(p(x, x) = p(x, y)$ and $p(y, y) = p(y, x))$ whenever $x, y \in X$.
	\end{enumerate}

	The pair $(X, p)$ will be called partial quasi-metric space.

	If $p$ satisfies all these conditions except possibly (1b), we shall speak of a \textit{lopsided partial quasi-metric type} or a \textit{lopsided partial quasi-metric}.
\end{definition}

\begin{remark}
	If $p$ is a partial quasi-metric on $X$ satisfying
	$(4)$ $p(x, y) = p(y, x)$ whenever $x, y \in X$, then $p$ is called a partial metric on $X$ in the sense of \cite{mat}.
\end{remark}

\begin{lemma} (\cite[Lemma 2.]{kun})
	
	\begin{enumerate}
		\item[(a)] Each quasi-metric $p$ on X is a partial quasi-metric on $X$ with $p(x, x) = 0$ whenever $x \in X$.
		\item[(b)] If $p$ is a partial quasi-metric on $X$, then so is its conjugate $p^{-1} (x, y) = p(y, x)$ whenever$ x, y \in X$.
		\item[(c)] If $p$ is a partial quasi-metric on $X$, then $p^+$ defined by $p^+(x, y) = p(x, y) + p^{-1}(x, y)$ is a partial $K$-metric on $X$.
	
	\end{enumerate}
	
\end{lemma}

The notions such as convergence, completeness, Cauchy sequence in the setting of partial
metric spaces, can be found in \cite{alg,mat} and references therein.

For every $K$-partial quasi-metric space $(X,p,K)$, the collection of balls $$p(x,\epsilon) = \{y \in X : p(x, y) < \epsilon + p(x, x)\}$$ yields a base for a $T_0$-Topology $\tau(p)$ on $X$.

Now, we define Cauchy sequence and convergent sequence in $K$-partial quasi-pseudometric spaces.

\begin{definition}
	Let $(X, p)$ be a partial quasi-metric space. Let $(x_n)_{n\geq 1}$ be any sequence in $X$ and $x \in X$. Then:
	
	\begin{enumerate}
		\item The sequence $(x_n)_{n\geq 1}$ is said to be convergent with respect to $\tau(p)$ (or $\tau(p)$-convergent) and converges to $x$, if $\lim\limits_{n\to \infty} p(x , x_n) = p(x, x)$.

		\item The sequence $(x_n)_{n\geq 1}$ is said to be convergent with respect to $\tau(p^+)$ (or $\tau(p^+)$-convergent) and converges to $x$, if $\lim\limits_{n\to \infty} p^+(x_n , x) = p^+(x, x)$.
		
		\item The sequence $(x_n)_{n\geq 1}$ is said to be a left $p$-Cauchy sequence if
		$$\lim\limits_{n\leq m,n,m\to \infty} p(x_n , x_m)$$ exists and is finite.
		
		\item The sequence $(x_n)_{n\geq 1}$ is said to be $\tau(p^+)$-Cauchy sequence if
		
		$$\lim\limits_{n,m\to \infty} p^+(x_n , x_m)$$ exists and is finite.
		
		\item $(X, p)$ is said to be $\tau(p)$-complete if for every $\tau(p)$-Cauchy
		sequence $(x_n)_{n\geq 1} \subseteq X$, there exists $x \in X$ such that:
		
		$$ \lim\limits_{n<m,n,m\to \infty} p(x_n , x_m)= \lim\limits_{n\to \infty} p(x , x_n)=p(x,x).$$

		\item $(X, p)$ is said to be left $p$-sequentially complete if every left $p$-Cauchy sequence converges for the topology $\tau(p)$.

		\item $(X, p)$ is said to be $p$-sequentially complete if every $\tau(p^+)$-Cauchy sequence is $\tau(p)$-convergent, i.e. there exists $x \in X$ such that:
		
		$$ \lim\limits_{n,m\to \infty} p^+(x_n , x_m)= \lim\limits_{n\to \infty} p(x , x_n)=p(x,x).$$
		
		\item $(X, p)$ is said to be $\tau(p)$-Smyth complete if every
		left $p$-Cauchy sequence is $\tau(p^+)$-convergent.

	\end{enumerate}
	
\end{definition}

\begin{remark}
	
	It is worthwhile here to point out the fact that, in this manuscript, the notation 
	
	$$\lim\limits_{n,m\to \infty} \qquad \text{ means ``simultaneously", i.e. } \lim\limits_{n\to \infty,m\to \infty}$$ 
	
	Indeed, in other contexts it means $
	\lim\limits_{n\to \infty} \lim\limits_{m\to \infty}$, for instance, which is clearly a different type of convergence in general .

Also, the following implications are easy to check:

$$ \tau(p)\text{-Smyth complete} \quad 
\Longrightarrow  \quad 
\text{left } p\text{-sequentially complete} \quad 
\Longrightarrow  \quad 
p\text{-sequentially complete}.$$
\end{remark}

\section{Main results}

Our proofs are inspired by the recent work of Romaguera et al. \cite{rom}.
	
\begin{definition}
	Let $(X, p)$ be a partial quasi-metric space. By a $p$-Kannan mapping on $X$, we mean a self-mapping $T$ on $X$ such that there exists a constant $0\leq \lambda < 1/4$ satisfying
		
		\begin{equation}\label{eq2}
			p(T x, T y) \leq \lambda [p(x, Tx) + p(y,Ty)]
		\end{equation}
		for all $x, y \in X.$

\end{definition}

	\begin{lemma}\label{lemma2}
		Let $T$ be a $p$-Kannan mapping on the partial quasi-metric space $(X, p,K)$ with $0\leq \lambda<1/4$. Then:

		\begin{enumerate}

			\item[(a)] $p^+(T x, T y) \leq 2\lambda (p(x, T x) + p(y, T y))$, for all $x, y \in X.$
			
			\item[(b)] $T$ is a Kannan mapping on the partial metric space $(X,p^+)$ with $0\leq \gamma<1/2$, i.e.
			$$p^+(Tx,Ty) \leq \gamma [p^+(x,Tx) + p^+(y,Ty)], \quad \text{ for all } x, y \in X.$$
			
			\item[(c)] For any $x_0 \in X$, the sequence $(T^n x_0)_{n\geq 1}$, is $\tau(p^+)$-Cauchy sequence and $$\lim\limits_{n,m\to \infty}p^+(x_n,x_m)=0.$$
		\end{enumerate} 
	\end{lemma}
	
	\begin{proof}

		(a) Given $x,y \in X$, we have
		
		$$p(T x, T y) \leq \lambda [p(x, Tx) +p(y,Ty)] \quad \text{and} \quad p(T y, T x) \leq \lambda [p(y,Ty) + p(x, Tx)],$$
		
		so 
		
		$$p^+(T x, T y) \leq 2\lambda [p(x, Tx)+ p(y,Ty)], $$

		\vspace*{0.3cm}

		(b) Since $p(x,Tx) \leq p^+(x,Tx)$ and  $p(y,Ty) \leq p^+(y,Ty)$ for any $x,y \in X$, 
		
		$$p^+(T x, T y) \leq 2\lambda [p^+(x, Tx)+ p^+(y,Ty)]= \gamma [p^+(x, Tx)+ p^+(y,Ty)], $$
		where $\gamma = 2\lambda.$
		It follows that $T$ is a Kannan mapping on the partial metric space $(X,p^+)$ with $0\leq \gamma<1/2$.

		\vspace*{0.3cm}
		
		(c) Since $T$ is a Kannan mapping on the partial $K$-metric space $(X,p^+)$ (here $K=1$), the classical proof of Shukla's fixed point theorem \cite{shu} shows that for any $x_0 \in X$, the sequence $(T^n x_0)_{n\geq 1}$, is $\tau(p^+)$-Cauchy sequence and
		
		 $$\lim\limits_{n,m\to \infty}p^+(x_n,x_m)=0.$$
	\end{proof}

	\begin{theorem}\label{thm1}
		Let $(X, p)$ be a $p$-sequentially complete partial quasi-metric space. Then every $p$-Kannan mapping on $(X, p,K)$ with constant $0\leq \lambda<1/4$ has a unique fixed point $x^*$ and $p(x^*,x^*)=0.$
	\end{theorem} 
	
	\begin{proof}
	Let $T$ be a $p$-Kannan mapping on $(X, p)$. Then, there exists $c \in [0, 1/4)$
	such that the contraction condition \eqref{eq2} follows for all $x, y \in X.$ For any $x_0 \in X$ fixed, Lemma \ref{lemma2}
	guarantees that the sequence $(T^n x_0)_{n\geq 1}$ is a $\tau(p^+)$-Cauchy sequence in the partial metric space $(X, p^+)$. Since $(X, p)$ is $p$-sequentially complete, there exists $z \in X$ such that 
	
	\begin{equation}\label{cauchy}
		0= \lim\limits_{n,m\to \infty}p^+(x_n,x_m)=\lim\limits_{n\to \infty} p(z , x_n) = p(z, z).
	\end{equation}

	Next we show that $Tz$ is the unique fixed point of $T$. To this end, we first show
	that $p(z, Tz) = 0$. Indeed, we have
	
	\begin{align*}
	p(z, T z) &\leq  p(z, T^n x_0 ) + p(T^n x_0 , Tz) -p(T^n x_0 ,T^n x_0 ) \\
	  & \leq p(z, T^n x_0 ) + p(T^n x_0 , Tz)\\
	  & \leq p(z, T^n x_0 ) + \lambda [d(T^{n-1} x_0 , T^n x_0 ) + p(z, T z)],
	\end{align*}
for all $n \in \mathbb{N}$.	Using \eqref{cauchy}, we deduce that $p(z, T z) \leq \lambda p(z, T z)$. Consequently,
$p(z, T z) = 0$.

Since by Lemma \ref{lemma2}(a),
$$p^+ (Tz, T^2 z) \leq 2 \lambda (p(z, T z) + p(T z, T^2z )),$$	

we deduce that $d ^+ (T z, T 2 z) \leq c p(T z, T^2 z),$ so $p^+ (T z, T^2 z) = 0$, i.e., $Tz$ is a fixed point of $T$.

Now, let us show that if $u \in X$ is a fixed point of $T$, that is, $T u = u$, then $p(u, u) = 0.$

Indeed, from \eqref{eq2}, $p^+(u, u) = p^+(T u, T u) \leq \lambda [p(u, T u) + p(u, T u)] = 2\lambda p(u, u) < p(u, u),$
a contradiction. Therefore, we must have $p(u, u) = 0.$

Finally, if $T u = u$ , it follows from Lemma \ref{lemma2}(a) that

$$   p^+(u, T z) = p^+ (T u, T^2 z) \leq \lambda [p(u, T u) + p(T z, T^2 z)].$$

Since $p(u, T u) = p(T z, T^2 z) = 0$, we deduce that $p^+ (u, T z) = 0$, i.e., $u = T z$ and the fixed point is unique.

This
concludes the proof.

\end{proof}
	
	The following examples illustrate Theorem \ref*{thm1}.

	\begin{example}
	Let $X = [0, \infty)$ and let $p$ be the partial quasi-metric on X given by
	$p(x, y) = \max\{x - y, 0\}+x$ for all $x, y \in X$.
$(X, p)$ is $p$-sequentially complete (in fact, it is left $p$-sequentially complete because
	every left $p$-Cauchy sequence in $X$ converges to 0 for $\tau(p)$).

	Now define $T : X \to X$ as $T x = 0$ if $x \in [0, 1]$ and $T x = x/8$ if $x \in (1, \infty)$.	
	
Let $x,y\in X$, and assume, without loss of generality, that $x\leq y.$	We distinguish the three following cases:

\begin{itemize}
	\item Case 1: If $x,y \in [0,1], $ $p^+(Tx,Ty)=p^+(0,0)=0.$
	\item Case 2: If $x\in [0,1]$ and $y\in (1, \infty)$, we have 
	
	\[p^+(Tx,Ty) = p^+\left(0,\frac{y}{8}\right)=\frac{y}{4}\leq \frac{2}{15}
	\left( 2x+\frac{15}{8}y\right)= \frac{2}{15} (p(x,Tx)+p(y,Ty).\]
	
	\item Case 3: If $x,y\in (1, \infty)$, we have 
	
	\[ p^+(Tx,Ty) = p^+\left(\frac{x}{8},\frac{y}{8}\right)= \frac{y}{4}\leq \frac{2}{15} \left(\frac{15}{8}x + \frac{15}{8}y  \right)=\frac{2}{15} (p(x,Tx)+p(y,Ty). \]
	
\end{itemize}

Therefore $T$ is a $p$-Kannan mapping on $(X, p)$ for $c = 2/15$. Thus, all conditions of
Theorem \ref{thm1} are satisfied. In fact $z = 0$ is the unique fixed point of $T$.

	\end{example}

\begin{theorem}	
	A partial quasi-metric space
	(X, d) is $p$-sequentially complete if and
	only if every $p$-Kannan mapping on $(X, p)$ has a fixed point.
	\end{theorem}

\begin{proof}

Suppose that $(X, p)$ is $p$-sequentially complete. Then, every $p$-Kannan
mapping on $(X, p)$ has a (unique) fixed point by Theorem \ref{thm1}.

For the converse suppose that $(X, p)$ is not $p$-sequentially complete. Then there
exists a $\tau(p^+)$-Cauchy sequence $(x_n)_{n\geq 1}$ such that $ \lim\limits_{n,m\to \infty}p^+(x_n,x_m)=0$ that does not converge for $\tau(p)$. Then,
for each $x \in X$ there exists $n_x \in \mathbb{N}$ such that $p(x, x_n ) > 0,$ for all $n \geq n_x$ (indeed,
otherwise there is $x \in X$ such that for each $n \in \mathbb{N}$ we can find $m_n \geq n$ for
which $p(x, x_{m_n} ) = 0;$ since $(x_n))_{n\geq 1}$ is a $\tau(p^+)$-Cauchy sequence it follows that
$(x_n)_{n\geq 1}$ converges to $x$ for $\tau(p)$, a contradiction).

Now, for each $x \in X$ put $C_x = \{x n : n \geq n_x \}$ and observe that $p(x, C_x ):= \inf \{p(x,y), y \in C_x\} > 0.$ Since 
$(x_n)_{n\geq 1}$ is a $\tau(p^+)$-Cauchy sequence such that $ \lim\limits_{n,m\to \infty}p^+(x_n,x_m)=0$, for each $x \in X$ there exists
$n(x) \geq n_x$ such that 
$$p^+(x_n,x_m) < \frac{1}{4} p(x,C_x), \text{ for all }  m,n \geq n(x).$$

Define $T : X \to X$ as $T x = x_{n(x)}$ for all $x \in X$. Since $n(x) \geq n_x$, we have that $p(x, x_{n(x)})>0$, and hence $T$ has no fixed point.

To complete the proof, we shall show that, $T$ is a $p$-Kannan mapping on $(X, p)$ for $c =1/8$. Indeed, let $x, y \in X$ and suppose, without loss of generality, that $n(x) \leq n(y)$. Then

\begin{align*}
p^+(Tx,Ty) & = p^+(x_{n(x)},x_{n(y)}) < \frac{1}{8}p(x,C_x)\\
    &\leq \frac{1}{8}p(x,x_{n(x)}) = \frac{1}{8}p(x,Tx).
\end{align*}

Since $p(T x, T y) \leq p^+ (T x, T y)$ and $p(T y,T x) \leq p^+(T x, T y)$, we conclude that $T$
is a $p$-Kannan mapping on $(X, p)$ for $c = 1/8$. This contradiction finishes the
proof.

\end{proof}

	\bibliographystyle{amsplain}

\end{document}